\documentclass[10.5pt,journal,onecolumn]{IEEEtran}
\usepackage{ifpdf}
\usepackage{url}

\IEEEoverridecommandlockouts 
\usepackage[utf8]{inputenc}
\usepackage{amsmath}
\usepackage[pdftex]{graphicx}
\usepackage[usenames]{color}
\usepackage{setspace}
\usepackage{caption}
\usepackage{subcaption}
\usepackage{tikz}
\usepackage{bm}
\usepackage{xspace}
\usepackage
[
a4paper,
left=0.7in,
right=0.7in,
top=0.8in,
bottom=1.1in,
]
{geometry}
\usepackage{soul}
\usepackage{amsopn}
\usepackage{comment}
\usepackage{algpseudocode}
\usepackage{algorithm} 
\usepackage{enumerate}
\makeatletter
\let\NAT@parse\undefined
\makeatother
\usepackage{amsmath,amsfonts, amssymb,amsthm,color}
\usepackage[numbers,sort&compress]{natbib}
\usepackage{amsf onts}
\usepackage{color}
\usepackage{setspace}

\usepackage{amssymb}
\newtheorem{theorem}{Theorem}

\newtheorem{assumption}{Assumption}

\newtheorem{lemma}{Lemma}

\usepackage{breqn}
\usepackage{hyperref} 
\usepackage{epstopdf}
\usepackage{lipsum}
\usepackage{epsfig}
\DeclareGraphicsExtensions{.pdf,.eps,.png,.jpg,.mps}

\newcommand{\ME}[1]{\textcolor{blue}{\textbf{(ME: \# 1)}}}
\newcommand{\CE}[1]{\textcolor{red}{\textbf{[CE: \# 2]}}}

\newcommand{\R}{\mathbb{R}}

\DeclareMathOperator*{\argmin}{\arg\!\min}

\title{ {\LARGE \bf 
Sublinear Regret with Barzilai-Borwein Step Sizes}
}

\author{Iyanuoluwa Emiola  
\thanks{The author is with the Electrical \& Computer Engineering Department at the University of Central Florida, Orlando FL 32816, USA. Email: \texttt{iemiola@knights.ucf.edu  }}%
}
\begin{document}
\renewcommand\IEEEkeywordsname{Keywords}
\maketitle

\begin{abstract}
This paper considers the online scenario using the Barzilai-Borwein Quasi-Newton Method. In an online optimization problem, an online agent uses a certain algorithm to decide on an objective at each time step after which a possible loss is encountered. Even though the online player will ideally try to make the best decisions possible at each time step, there is a notion of regret associated with the player’s decisions. This study examines the regret of an online player using optimization methods like the Quasi-Newton methods, due to their fast convergent properties. The Barzilai-Borwein (BB) gradient method is chosen in this paper over other Quasi-Newton methods such as the Broyden-Fletcher-Goldfarb-Shanno (BFGS) algorithm because of its less computational complexities. In addition, the BB gradient method is suitable for large-scale optimization problems including the online optimization scenario presented in this paper. To corroborate the effectiveness of the Barzilai-Borwein (BB) gradient algorithm, a greedy online gradient algorithm is used in this study based on the two BB step sizes. Through a rigorous regret analysis on the BB step sizes, it is established that the regret obtained from the BB algorithm is sublinear in time. Moreover, this paper shows that the average regret using the online BB algorithm approaches zero without assuming strong convexity on the cost function to be minimized.
 \end{abstract}
\section{Introduction}
This paper presents a gradient-based algorithm using the Barzilai-Borwein step sizes to solve an online optimization problem with regret analysis of an online player. \textit{Online Optimization} involves a process where an online agent makes a decision without knowing whether the decision is correct or not. The objective of the online agent is to  make a sequence of accurate decisions given knowledge of the optimal solution to previous decisions. A common term associated with many optimization problems known as \textit{regret} measures how well the online agent performs after a certain time, based on the the difference between the loss incured and the best decision taken \cite{hazan2016introduction}.  The problem of online optimization has applications to a number of fields including game theory, the smart grid and classification in machine learning amongst others. 
Performance of online optimization algorithms is usually measured in terms of the aggregate regret suffered by the online agent compared with the known optimal solution of each problem across the sequence of problems.

Online optimization methods and algorithms have been studied using different methods including gradient-based methods \cite{shalev2009mind,zinkevich2003online,hazan2016introduction}. 
Extensions have been considered on unconstrained problems \cite{mcmahan2012no} and online problems with long-term \cite{mahdavi2012trading}. Problems in dynamic environments have also been analyzed in \cite{mokhtari2016online}, \cite{zhang2019distributed}, \cite{yi2020distributed}, \cite{sharma2020distributed},  \cite{lesage2020second} and \cite{chen2017online}. The author in \cite{zhang2019distributed} used gradient tracking technique in a static optimization scenario and showed that the regret bounds in the dynamic optimization case is independent of the time horizon. In \cite{yi2020distributed}, the authors obtained sublinear regret in a dynamic case  for a distributed online problem using the primal-dual descent algorithm. The authors in \cite{sharma2020distributed} obtained sublinear regret for a distributed online framework that has time-varying constraints and presented a fit technique to deal with constraint violations. In \cite{chen2017online}, the authors applied the online optimization problem with an application to adversarial attack. The authors explored an online constrained problem with adversarial objective functions and constraints and obtained a sublinear regret. In addition, the authors in \cite{mokhtari2016online}, \cite{zhang2019distributed}, \cite{yi2020distributed}, \cite{sharma2020distributed},  \cite{lesage2020second} and \cite{chen2017online} used gradient methods in their computational approach to establish convergence. As well-structured as gradient methods are, applying them to large-scale online problems face several challenges and become impractical due to their well-known slow convergence rates in the static settings \cite{boyd2004convex}. 
To address the slow convergence rates of first order methods, second-order (popularly called Newton-type) methods have been proposed \cite{schraudolph2007stochastic}. The Newton method was applied in \cite{lesage2020second} where the authors showed that the Newton method performs similarly to a case where the strong convexity condition is used on the objective function. While Newton-type iterative methods have quadratic convergence, they also present a significant computational overhead from the need to invert and store the hessian of the objective function being optimized, which makes them impractical for large-scale online optimization problems. This paper aims at even improving the Newton method with the Barzilai-Borwein Quasi-Newton methods in an online optimization scenario.

To leverage the benefits of the computational simplicity of gradient methods and the convergence properties of second-order methods, the so-called quasi-Newton methods have been introduced; for example, the Broyden-Fletcher-Goldfarb-Shanno (BFGS) algorithm \cite{li2001global,eisen2017decentralized} and the Barzilai-Borwein (BB) algorithm \cite{barzilai1988two,dai2005projected}.
Quasi-Newton methods exploit the second-order (curvature information) of the objective function being optimized into the first-order framework. For example, the BFGS method approximates the information in the curvature of the hessian between time steps to use in its update, though scaling is a known issue \cite{ibrahimscaling}. The Stochastic BFGS and its low-memory variant (the L-BFGS) quasi-Newton method has been studied in online settings \cite{schraudolph2007stochastic,byrd2016stochastic} with good performance relative to the standard gradient method. The BB method, on the other hand, computes a step size such that the computed step size and gradient contain information that approximates the hessian curvature. Convergence rate analyses have been obtained for these quasi-Newton methods \cite{byrd1987global,dai2013new}
and these methods are increasingly being used in large-scale, computation-intensive applications such as distributed learning.
In this paper, an online Barzilai-Borwein quasi-Newton algorithm is presented and its performance for the two variations of the BB step sizes is analyzed using the regret. We show that the regret increases sublinearly in time.
Following an introduction of the problem and brief summary of existing approaches (Section \ref{sec:problem form-BB}), we introduce quasi-Newton methods that exploit known fast convergence of second-order methods (Section \ref{sec:about-BB}) and present our main result (Section \ref{sec:convergenceregret-BB}). Concluding remarks follow in Section \ref{sec:conclude}.

\subsection{Contributions}
In this paper, a novel regret analysis using the Barzilai-Borwein Quasi-Newton method in an online optimization scenario is presented. Due to the fast convergence property of the Newton methods, the work \cite{lesage2020second} is an improvement on existing online optimizations application problems in \cite{zhang2019distributed}, \cite{yi2020distributed}, \cite{sharma2020distributed}, and \cite{chen2017online}. However, the Quasi-Newton method using the BB step sizes presented in this paper is better than Newton  methods in dealing with convergence speeds and computing the inverse of the hessian. Even though the author in \cite{zinkevich2003online} also obtained a similar sublinear regret result, BB Quasi-Newton algorithm is known to be suitable for dealing with large-scale optimization bottleneck that the Newton method is not appropriate for. Additionally, strong convexity assumption is not needed in this paper to establish sublinear regret.

 \noindent \textit{Notation:} We represent vectors and matrices as lower and upper case letters, respectively. Let a vector or matrix transpose be $(\cdot)^T$, and the L$2$-norm of a vector be $\|\cdot\|$. Let the gradient of a function $f(\cdot)$ be $\nabla f(\cdot)$, and the set of reals numbers be $\mathbb{R}$.

\section{Problem formulation}\label{sec:problem form-BB}
Consider an online optimization problem
\begin{equation}\label{eqn:cost}
\min_{x(k)\in \mathcal{X}} f_k(x(k)),
\end{equation}
in which the feasible decision set $\mathcal{X} \in \mathbb{R}^n$ is known, assumed to be convex quadratic, non-empty, bounded, closed and fixed for all time $k = 1\hdots, K$. We assume  the number of iterations during which the online players make choices, $K$, is unknown to the player. By convexity of the cost function $f_k(\cdot)$ and $\mathcal{X}$, Problem \eqref{eqn:cost} has an optimal solution $x^*$, which is the best possible choice or decision agents can make at each time $k$.  A player (an online agent) at time $k$ uses some algorithm to choose a point $x(k) \in \mathcal{X}$, after which the player receives a loss function $f_k(\cdot)$. The loss incurred by the player is $f_k(x(k))$. These problems are common in contexts such as real time resource allocation, online classification \cite{hazan2016introduction}. The goal of the online agent is to minimize the aggregate loss by determining a sequence of feasible online solutions $x(k)$ at each time-step of the algorithm. 

Let the aggregate loss incurred by the online algorithm that solves Problem \eqref{eqn:cost} at time $K$ be given by:
\begin{equation*}
    f(K)=\sum_{k=1}^{K}f_k(x(k)).
\end{equation*}
To measure performance of the online player, we use the regret framework. The \textit{static regret} is a measure of the difference between the loss of the online player and the loss from the static case
\begin{equation*}
\min_{x\in \mathcal{X}} f_k(x),
\end{equation*}
where the single best decision $x^*$ is chosen with the benefit of hindsight. Let the aggregate loss up to time $K$ incurred by the single best decision be given by 
\begin{equation*}
    f_x(K)=\sum_{k=1}^{K}f_k(x).
\end{equation*}
Then the static regret at time $K$ is defined as \cite{hazan2016introduction}:
\begin{equation}
  R(K)= f(K)- \min_{x\in \mathcal{X}} f_x(K). 
\end{equation}

\subsection{Algorithms for Online Optimization Problem}


A commonly used algorithm for solving the static case of Problem \eqref{eqn:cost} is the gradient descent method, which involves updating the variable $x(k)$ iteratively  using the gradient of the cost function with the following equation:
\begin{equation}\label{eqn:gradient}
x(k+1) = x(k) - \alpha \nabla f(x(k)).
\end{equation}
It is known that with an appropriate choice of the step size $\alpha$, the sequence $\{x(k)\}$ converges to $x^*$ in $\mathcal{O}(1/k)$; that is, an $\varepsilon$-optimal solution is attained in about $\mathcal{O}(\frac{1}{\varepsilon})$ iterations \cite{nesterov1998introductory}.
%
Moreover, when the cost function is strongly convex, the update equation in \eqref{eqn:gradient} reaches an $\varepsilon$-optimal solution in about $\mathcal{O}(1/\varepsilon^2)$ iterations.  Even though the update scheme of gradient method are easily implementable in a distributed architecture as seen in \cite{nesterov1998introductory} and \cite{emiola2021distributed}, there have been a need for an improvement in   convergence rates of gradient methods as seen in \cite{su2014differential}. Nonetheless techniques to accelerate convergence lag behind the Newton and quasi-Newton methods \cite{su2014differential}.

To improve convergence rates in static optimization problems, algorithms that use second order information (hessian of the cost function) have been introduced. These methods leverage curvature information of the cost function in addition to direction; and are known to speed up the convergence in the neighborhood of the optimal solution. The Newton-type method is an example used as an improvement in enabling faster convergence rates than the regular gradient method.
In fact, when the cost function is quadratic, the Newton algorithm is known to converge in one time-step. For non-quadratic, the Newton method still converges in just a few time steps \cite{emiola2021comparison}. Though they have good convergence properties, there are computational costs associated with building and computing the inverse hessian. In addition, some modification are needed if the hessian is not positive definite \cite{gill1972quasi}. To avoid the computation burden of second-order methods while maintaining the structure of first-order methods, Quasi-Newton methods have been introduced. 

\subsection{Quasi-Newton Methods}
A number of quasi-Newton methods have been proposed in the literature including the Broyden-Fletcher-Goldfarb-Shanno (BFGS) algorithm \cite{eisen2017decentralized} and the Barzilai-Borwein (BB) algorithm \cite{dai2002r}, as well as the David-Fletcher-Powell (DFP) algorithm \cite{sofi2013reducing}. The main idea in the performance of these methods is to speed up convergence by using the information from the inverse hessian without computing it explicitly; for example, Barzilai-Borwein computes step-sizes using the difference of successive iterates and the gradient evaluated at those iterates. 
Although the BFGS can be used to facilitate rapid convergence, scaling is an issue especially during the process where the method approximates the information in the curvature of the Hessian between time steps as seen in \cite{ibrahimscaling}. However the Barzilai-Borwein Quasi-Newton method uses just the step sized to approximate the inverse hessian.
In this paper, we use the gradient-based method using Barzilai-Borwein step sizes to solve Problem \eqref{eqn:cost} and show that the regret increases sublinearly in time. 

\section{The Barzilai-Borwein Quasi-Newton Method}\label{sec:about-BB}
The Barzilai-Borwen quasi-Newton method is an iterative technique suitable for solving optimization problems that can yield superlinear convergence rates when the objective functions are strongly convex and quadratic \cite{barzilai1988two,dai2013new}.
 It differs from other quasi-Newton methods because it only uses one step size for the iteration as opposed to other quasi-Newton method that have more computation overhead. The Barzilai-Borwein method solves Problem \eqref{eqn:cost} iteratively using the update in \eqref{eqn:gradient}; however, the step-size $\alpha(k)$ is computed so that $\alpha(k) \nabla f(x(k))$ approximates the the inverse Hessian. We briefly introduce the two forms of the BB step-sizes used in Algorithm~\ref{alg:algorithm}. 

Consider the update $x(k+1) = x(k) - \alpha(k) \nabla f(x(k))$. 
The two forms of the BB step sizes \cite{barzilai1988two} $\alpha_1(k)$ and $\alpha_2(k)$ are given by:
\begin{equation}\label{eqn:bb-step}
    \alpha_1(k)=\frac{s(k-1)^{T} s(k-1)}{s(k-1)^T y(k-1)}.
\end{equation}
%
\begin{equation}\label{eqn:bb-step1}
     \alpha_2(k)=\frac{s(k-1)^{T} y(k-1)}{y(k-1)^T y(k-1)}.
\end{equation}
and $s(k)$ and $y(k)$ are such that 
\[
s(k-1)\triangleq x(k)-x(k-1), \quad \text {and}
\]
\begin{equation}\label{eqn:8}
    y(k-1)=\nabla f(x(k))-\nabla f(x(k-1)). \nonumber
\end{equation}
In general, there is flexibility in the choice to use $\alpha_1(k)$ or $\alpha_2(k)$ \cite{barzilai1988two}, and both step sizes can be alternated within the same algorithm after a considerable amount of iterations to facilitate convergence. The rest of this work will characterize performance of the online Algorithm \ref{alg:algorithm} using the step sizes in Equations \eqref{eqn:bb-step} and \eqref{eqn:bb-step1}, which as we will show has a regret that is sublinear in time with the average regret approaching zero.

Before stating the main result, we state some assumptions about Problem \eqref{eqn:cost} and Algorithm \ref{alg:algorithm}.
\begin{assumption}\label{assume-bounded-G}
The decision set $\mathcal{X}$ is bounded. This implies that there exists some constant $0 \leq B < \infty$ such that $|\mathcal{X}|\leq B$.
\end{assumption}

\begin{assumption}\label{assume-closed-set}
The decision set $\mathcal{X}$ is closed; that is, suppose all agents'  decisions follow an iterative sequence $x(k) \in \mathcal{X}$. If there exists some $\hat{x}\in \mathbb{R}^n$ such that $\lim_{k\rightarrow\infty} x(k) = \hat{x}$, then $\hat{x}\in \mathcal{X}$.
\end{assumption}

\begin{assumption}\label{assume-bounded-gradient}
For all decision iterates $x(k)$, the cost function $f(x(k))$ is differentiable and the gradient of the objective function $\nabla f$ is Lipschitz continuous. This means that for all $x$ and $y$, there exists $L>1$ such that:
\begin{equation*}
   \|\nabla f(x)-\nabla f(y)\|\leq L\|x-y\|.
\end{equation*}
\end{assumption}

\begin{algorithm}
\caption{Online Barzilai-Borwein Quasi-Newton Alg.}
\label{alg:algorithm}
\begin{algorithmic}[1]
\Statex \textbf{Given}: Feasible set $\mathcal{X}$ and time horizon $K$
\Statex \textbf{Initialize}: $x(0)$ and $\nabla f_0(x(0)$ arbitrarily
\For {$k=1$ to $K$}
\State Agents predicts $x(k)$ and observes $f_k(\cdot)$
\State Update $x(k+1) =  x(k) -  \alpha(k) \nabla f_k(x(k))$
\EndFor
\end{algorithmic}
\end{algorithm}


\section{Regret Bounds}
\label{sec:convergenceregret-BB}

Before we present our main results (Theorems \ref{thm:main-result} and \ref{thm:main-result-2}), we first present two lemmas that will be used in its proof. The first is a result in \cite{zinkevich2003online}, which will be used in the definition of regret and the other is the Sedrakyan's inequality.

\begin{lemma}(\cite{zinkevich2003online})\label{lemma-regret}
Without loss of generality, for all iterates $k$, there exists gradient $g(k)\in \mathbb{R}^n$ such that for all $x$, $g_k.x=f_k(x)$, where $g_k=\nabla f_k(x(k))$.\end{lemma}
\begin{proof}
The proof can be seen in \cite{zinkevich2003online}.
\end{proof}
\begin{lemma} (\textbf{The Sedrakyan's Inequality})\label{Titu's-Lemma}
 For all positive reals $a_1, a_2,........a_n$ and  $b_1, b_2,........b_n$, the following inequality holds:
\begin{equation*}
    \sum_{i=1}^{n}\frac{a_i^{2}}{b_i}\geq \frac{(\sum_{i=1}^{n}a_i)^{2}}{\sum_{i=1}^{n}b_i}.
\end{equation*}
\end{lemma}
\begin{proof}
We refer readers to \cite{sedrakyan2018algebraic} for a proof.
\end{proof}
\noindent Another result we will use is  the static regret bounds for $R(K)$ which is shown in \cite{zinkevich2003online}: 
 \begin{equation}\label{eqn:generalized}
     R(K)\leq \|D\|^2\frac{1}{2\alpha (K)}+\frac{\|\nabla f_m\|^2}{2} \sum_{k=1}^{K}\alpha (k), 
  \end{equation}
As seen in \cite{zinkevich2003online}, $D$ denotes the maximum value of the diameter of  $\mathcal{X}$ and $ \|\nabla f_m\|=   \max_{x\in \mathcal{X}} \|\nabla f_k(x) \|$.

We will now proceed to characterize the regret obtained from Algorithm \ref{alg:algorithm} for Problem \eqref{eqn:cost} with the two BB step sizes.

\begin{theorem}\label{thm:main-result}
Consider Problem \eqref{eqn:cost} and let: 
$$\alpha (k)  =  \frac{s(k{-}1)^{T} s(k{-}1)}{s(k{-}1)^T y(k{-}1)}$$ in Algorithm \ref{alg:algorithm}. If $\frac{e-d}{c-b}\leq \frac{d}{b}$, where
\begin{equation*}
    b =(\|(x(1){-}x(0)\|+\|(x(2){-}x(1)\|)^2 ,
\end{equation*}
\begin{equation*}
    c = 2(\|(x(1){-}x(0)\|^2+\|(x(2){-}x(1)\|^2),
\end{equation*}
\begin{equation*}
  d =  \sum_{k=1}^{K}(x(k)-x(k-1))^{T} (\nabla f(x(k))-\nabla f(x(k-1))), 
\end{equation*}
and $e =  \sum_{k=1}^{K}L\|x(k)-x(k-1)\|^2$.

Also if $P=\min (P,Z)$ where:
\begin{equation*}
 P =   \sum_{k=1}^{K}\alpha (k)
\end{equation*}
and
\begin{equation*}
    Z = \frac{2(\|(x(1){-}x(0)\|^2+\|(x(2){-}x(1)\|^2)}{L \sum_{k=1}^{K}(\|x(k)\|^2+\|x(k-1)\|^2)} 
\end{equation*}
Then the
average regret is bounded by:
\begin{equation*} 
 \frac{R(K)}{K}\leq \|D\|^2\frac{1}{2K\alpha (K)}{+}\frac{\|\nabla f_m\|^2}{2K} \ \Psi,
 \end{equation*}
\[
\text{where} \ \ \Psi = \frac{2(\|(x(1){-}x(0)\|^2+\|(x(2){-}x(1)\|^2)}{L \sum_{k=1}^{K}\|x(k)\|^2{+}L\sum_{k=1}^{K}\|x(k{-}1)\|^2},
\] 
$L = \max_k L_k$, $L_k$ is the Lipschitz parameter of $\nabla f_k(x(k)$,
in Problem \eqref{eqn:cost}
and $\lim_{K\to\infty}\frac{R(K)}{K}$ approaches $0$. 
\end{theorem}

\begin{proof}
First, by using the results of Lemma \ref{lemma-regret}, the regret of Algorithm \ref{alg:algorithm} can be expressed as:
 \[ R(K)=   \sum_{k=1}^{K}(x(k)-x^*) g(k).\]
Then from Equation \eqref{eqn:gradient}, the regret
\[
\R(K)=   \sum_{k=1}^{K}(x(k-1) - \alpha(k-1) \nabla f(x(k-1))-x^*) g(k),
 \]
where $\alpha(k)$ is as expressed in  \eqref{eqn:bb-step} . To prove Theorem \ref{thm:main-result}, the approach will be to upper-bound the aggregate sum of the step size $\alpha(k)$ and use the generalized bound for online gradient descent in Equation \eqref{eqn:generalized}. This approach is possible since the gradient of the cost function at each time in the sequence of problems is bounded (Assumption \ref{assume-bounded-gradient}). Proceeding, the running sum of the step sizes $\alpha(k)$ up to time $K$ is expressed as
\begin{align*}
\sum_{k=1}^{K}\alpha (k)   &=   \sum_{k=1}^{K}\frac{s(k-1)^{T} s(k-1)}{s(k-1)^T y(k-1)} \\
&= \sum_{k=1}^{K}\frac{(x(k)-x(k-1))^{T}(x(k)-x(k-1))}{(x(k){-}(k{-}1))^{T} (\nabla f(x(k)){-}\nabla f(x(k{-}1)))}  \\
&=\sum_{k{=}1}^{K}\frac{\|x(k){-}x(k{-}1)\|^{2}}{(x(k){-}x(k{-}1))^{T}(\nabla f(x(k)){-}\nabla f(x(k{-}1)))}.
\end{align*}
By applying the result in Lemma \ref{Titu's-Lemma} to the right hand side of the preceding inequality, we obtain that:
\begin{equation} \label{eqn:titu-out}
 \sum_{k=1}^{K}\alpha (k) {\geq} \frac{(\sum_{k=1}^{K}\|(x(k){-}x(k{-}1))\|)^2}{\sum_{k=1}^{K}(x(k){-}x(k{-}1))^{T} (\nabla f(x(k)){-}\nabla f(x(k{-}1)))}   
\end{equation}
By inspection, if write the first few terms of the numerator of equation \eqref{eqn:titu-out}, it is evident that equation \eqref{eqn:titu-out} can be further lower bounded according to the following:
\begin{equation} \label{eqn:titu-out1}
 \sum_{k=1}^{K}\alpha (k) {\geq} \frac{(\|(x(1){-}x(0)\|+\|(x(2){-}x(1)\|)^2}{\sum_{k=1}^{K}(x(k){-}x(k{-}1))^{T} (\nabla f(x(k)){-}\nabla f(x(k{-}1)))}  
\end{equation}
Clearly because the terms $\|(x(1){-}x(0)\|$ and $\|(x(2){-}x(1)\|$ are positive, the numerator of equation \eqref{eqn:titu-out1} can be upper-bounded according to the following:
\begin{align*}
   (\|(x(1){-}x(0)\|+\|(x(2){-}x(1)\|)^2 \\ \leq  2(\|(x(1){-}x(0)\|^2+\|(x(2){-}x(1)\|^2)
\end{align*}
To bound the denominator of Equation \eqref{eqn:titu-out1}, we use the Lipschitz continuity of the gradients of $f(\cdot)$ with parameter $L>1$.
Therefore,
\begin{align*}
&\sum_{k=1}^{K}(x(k)-x(k-1))^{T} (\nabla f(x(k))-\nabla f(x(k-1))) \\
&\leq \sum_{k=1}^{K}L\|x(k)-x(k-1)\|^2.
\end{align*}
If we represent the bounds in the numerator and denominator of equation \eqref{eqn:titu-out1} by the following variables such that:
\begin{equation*}
    b =(\|(x(1){-}x(0)\|+\|(x(2){-}x(1)\|)^2 ,
\end{equation*}
\begin{equation*}
    c = 2(\|(x(1){-}x(0)\|^2+\|(x(2){-}x(1)\|^2),
\end{equation*}
\begin{equation*}
  d =  \sum_{k=1}^{K}(x(k)-x(k-1))^{T} (\nabla f(x(k))-\nabla f(x(k-1))),
\end{equation*}
and
\begin{equation*}
    e =  \sum_{k=1}^{K}L\|x(k)-x(k-1)\|^2.
\end{equation*}
It has been shown that $b\leq c$ and $d\leq e$. Therefore to find an upper bound for equation \eqref{eqn:titu-out1}, we use the condition that if $\frac{e-d}{c-b}\leq \frac{d}{b}$, then we obtain:
\begin{equation*}
    \frac{b}{d}\leq \frac{c}{e}
\end{equation*}
So we obtain the bounds of the right hand side of \eqref{eqn:titu-out1} as: 
\begin{align*}
&\frac{(\|(x(1){-}x(0)\|+\|(x(2){-}x(1)\|)^2}{\sum_{k=1}^{K}(x(k){-}x(k{-}1))^{T} (\nabla f(x(k)){-}\nabla f(x(k{-}1)))}   \\
&\leq \frac{2(\|(x(1){-}x(0)\|^2+\|(x(2){-}x(1)\|^2)}{\sum_{k=1}^{K}L\|x(k)-x(k-1)\|^2} \\
&\leq \frac{2(\|(x(1){-}x(0)\|^2+\|(x(2){-}x(1)\|^2)}{L \sum_{k=1}^{K}(\|x(k)\|^2+\|x(k-1)\|^2-2\|x(k)\|\|x(k-1)\|)} \\
&\leq \frac{2(\|(x(1){-}x(0)\|^2+\|(x(2){-}x(1)\|^2)}{L \sum_{k=1}^{K}(\|x(k)\|^2+\|x(k-1)\|^2)} 
\end{align*}
If we let the left hand side of equation \eqref{eqn:titu-out1} be represented by:
\begin{equation*}
 P =   \sum_{k=1}^{K}\alpha (k)
\end{equation*}
and we let the right hand side of equation \eqref{eqn:titu-out1} be denoted as:
\begin{equation*}
 Q =  \frac{(\|(x(1){-}x(0)\|+\|(x(2){-}x(1)\|)^2}{\sum_{k=1}^{K}(x(k){-}x(k{-}1))^{T} (\nabla f(x(k)){-}\nabla f(x(k{-}1)))}  
\end{equation*}
Similarly if we let the derived upper bound of $Q$ be given by:
\begin{equation*}
    Z = \frac{2(\|(x(1){-}x(0)\|^2+\|(x(2){-}x(1)\|^2)}{L \sum_{k=1}^{K}(\|x(k)\|^2+\|x(k-1)\|^2)} 
\end{equation*}
From the above analysis, we observe that $P\geq Q$ and $Q\leq Z$. Therefore, if $P=\min (P,Z)$, then we can deduce that $P\leq Z$.\\
By the established relationship between $P$ and $Z$ and also using the triangle inequality, we obtain the bound for using the first BB step size as: 
\begin{align*}
   \sum_{k=1}^{K}\alpha (k)\leq \frac{2(\|(x(1){-}x(0)\|^2+\|(x(2){-}x(1)\|^2)}{L\sum_{k=1}^{K}\|x(k)\|^2+L\sum_{k=1}^{K}\|x(k{-}1)\|^2}
\end{align*}
By using the regret bound equation in \eqref{eqn:generalized}, we obtain:
\begin{align*}
R(K)\leq \|D\|^2\frac{1}{2\alpha (K)}{+}\frac{\|\nabla f_m\|^2}{2} \Psi,
\end{align*}
$$ \text{where} \ \ \Psi = \frac{2(\|(x(1){-}x(0)\|^2+\|(x(2){-}x(1)\|^2)}{L \sum_{k=1}^{K}\|x(k)\|^2{+}L\sum_{k=1}^{K}\|x(k{-}1)\|^2}. $$
The average regret over $K$ time steps can then be expressed as
\begin{align*}
 \frac{R(K)}{K}\leq \|D\|^2\frac{1}{2K\alpha (K)}{+}\frac{\|\nabla f_m\|^2}{2K} \Psi.
\end{align*}
Since $\|D\|$ is constant based on its value in \eqref{eqn:generalized}, and $\|\nabla f_m\|^2$ is also constant, we conclude that that the average regret

$\lim_{K\to\infty}\frac{R(K)}{K}$ approaches $0$.
\end{proof}
Next, we consider the performance of Algorithm \ref{alg:algorithm} using the second BB step-size in Equation \eqref{eqn:bb-step1}.
\begin{theorem}\label{thm:main-result-2}
Consider Problem \eqref{eqn:cost} and let Algorithm \ref{alg:algorithm} be used to solve Problem \eqref{eqn:cost} where $\alpha (k)= \frac{s(k-1)^{T} y(k-1)}{y(k-1)^T y(k-1)}$; and $L$ is the maximum of all Lipschitz continuity parameters of all gradients of the cost function in Problem \eqref{eqn:cost}, then, the regret is upper bounded by
\begin{equation*}
R(K)\leq \|D\|^2\frac{1}{2\alpha (K)}+\frac{\|\nabla f_m\|^2}{2} \zeta,
\end{equation*}
where 
\[
\zeta = (\sum_{k{=}1}^{K} (((A(k)^T)^{2})^{\frac{1}{2}}(\sum_{k{=}1}^{K} ((B(k))^{2})^{\frac{1}{2}}(\sum_{k{=}1}^{K} ((C(k))^{2})^{\frac{1}{2}}.
\]
and the average regret
$\lim_{K\to\infty}\frac{R(K)}{K}$ approaches $0$.
\end{theorem}
\begin{proof}
The approach to proving Theorem \ref{thm:main-result-2} will be similar to that of Theorem \ref{thm:main-result}, where we will obtain bounds for the aggregate sum of the step sizes in $R(K)$ and use the generalized bound for online gradient descent algorithm. In this case, the sum of the aggregate step sizes is expressed as
\begin{equation*}
\sum_{k=1}^{K}\alpha (k)=\sum_{k=1}^{K}\frac{s(k-1)^{T} y(k-1)}{y(k-1)^T y(k-1)}
\end{equation*}
By using the relationship
\[
s(k-1)\triangleq x(k)-x(k-1), \quad \text {and}
\]
\begin{equation}\label{eqn:8}
    y(k-1)=\nabla f(x(k))-\nabla f(x(k-1)). \nonumber
\end{equation}
and by noting that $y(k-1)^T y(k-1)=\|y(k-1)\|^{2}$, and also expressing as a product of three different functions, we obtain:
\begin{equation}\label{eqn:splitting-step-sum}
\sum_{k=1}^{K}\alpha(k)=\sum_{k=1}^{K}((x(k){-}x(k-1))^T (\nabla f(x(k)){-} \nabla f(x(k{-}1)))\|\nabla f(x(k)){-}\nabla f(x(k{-}1)\|^{-2})
\end{equation}
For the purpose of clarity, let 
\begin{align*}
    A(k) &= ((x(k){-}x(k-1)) \\
    B(k) &= (\nabla f(x(k)){-} \nabla f(x(k{-}1))) \quad \text{and}\\
    C(k) &=     \|\nabla f(x(k)){-}\nabla f(x(k{-}1)\|^{-2}
\end{align*}
Applying the Cauchy-Schwarz inequality to the right hand side of Equation \eqref{eqn:splitting-step-sum}, we obtain that:
\begin{align*}
   \sum_{k{=}1}^{K}\alpha(k) &= \sum_{k{=}1}^{K} (A(k)^TB(k))C(k) \\
   &\leq \sum_{k{=}1}^{K} (((A(k)^T)^{2}(B(k))^{2})  \sum_{k{=}1}^{K} (C(k))^{2})^{\frac{1}{2} }, \\
   &\leq (\sum_{k{=}1}^{K} (((A(k)^T)^{2})^{\frac{1}{2}}(\sum_{k{=}1}^{K} ((B(k))^{2})^{\frac{1}{2}}(\sum_{k{=}1}^{K} ((C(k))^{2})^{\frac{1}{2}}.
\end{align*}

Applying the generalized regret bound as seen in Equation \eqref{eqn:generalized}, we obtain the regret $R(K)$ as: 
\begin{equation*}
R(K)\leq \|D\|^2\frac{1}{2\alpha (K)}+\frac{\|\nabla f_m\|^2}{2} \zeta,
\end{equation*}
where the value of $\zeta$ is the upper bound of $\sum_{k=1}^{K}\alpha(k)$ obtained above after applying Cauchy-Schwarz inequality and it is given by:
\[
\zeta = (\sum_{k{=}1}^{K} (((A(k)^T)^{2})^{\frac{1}{2}}(\sum_{k{=}1}^{K} ((B(k))^{2})^{\frac{1}{2}}(\sum_{k{=}1}^{K} ((C(k))^{2})^{\frac{1}{2}}.
\]
Therefore the average regret is
\begin{equation*}
     \frac{R(K)}{K}\leq \|D\|^2\frac{1}{2K\alpha (K)}+\frac{\|\nabla f_m\|^2}{2K} \zeta 
 \end{equation*}
Furthermore, since $\|D\|$ is constant based on its value in \eqref{eqn:generalized}, and the terms $A(k)$, $B(k)$ and $C(k)$ are also positive, we conclude that the average regret 
$\lim_{K\to\infty}\frac{R(K)}{K}$ approaches $0$.

\end{proof}
The Barzilai-Borwein step size in the gradient-based Algorithm \ref{alg:algorithm} results in a regret that grows sublinearly in time and yields an average regret of zero as time $K$ goes to infinity.



\section{Conclusions}
In this work, an online Barzilai-Borwein quasi-Newton algorithm using the regret framework is presented to show the usefulness of Quasi-Newton methods for large-scale and computational intensive optimization problems. The analysis for both Barzilai-Borwein step sizes showed that the regret of the algorithm grows sublinearly in time and that the average regret approaches zero. The use of the generalized regret bounds for online gradient descent introduced in \cite{hazan2016introduction} simplified the analyses. For future research work, a regret analysis in a dynamic scenario for online Quasi-Newton method will be presented using the Barzilai-Borwein and the Broyden-Fletcher-Goldfarb-Shanno (BFGS) algorithm. Another interesting optimization method with a fast convergence property is the \textit{Conjugate Gradient} method. It should perform well in an online optimization problem but it is unknown whether it will be superior to most online optimization algorithms. Therefore, readers are free to explore research topics on applying the conjugate gradient method to improve the convergence rates for online optimization problems.

\label{sec:conclude}

\bibliographystyle{IEEEtran}
{\small
\bibliography{dbb-ref-1.bib}}

\end{document}